\documentclass[11pt]{article}

\newif\iffullversion

\fullversiontrue

\usepackage{amsmath,amssymb,amsthm}
\usepackage[
  margin=1in,
  includefoot,
  footskip=30pt,
]{geometry}
\let\chapter\section
\usepackage[ruled,vlined,linesnumbered]{algorithm2e}
\usepackage{color}
\usepackage{enumitem}
\usepackage{graphicx} 
\setlist{leftmargin=*}
\usepackage{subfigure}
\newcommand{\eh}[1] {}
\newcommand{\jake}[1] {}

\def\E{\mathop{\mathbb{E}}}

\def\reals{\mathbb{R}}
\def\K{\mathcal{K}}

\def\bx{\bar{x}}
\def\hx{\hat{x}}
\def\vol{\textnormal{vol}}
\def\diam{\textnormal{diam}}

\def\tht{\hat{\theta}}

\def \hp{\textnormal{\sc HeatPath}}
\def\cp{\textnormal{\sc CentralPath}}

\def\har{\textnormal{\sc HitAndRun}}
\def\newton{\textnormal{\sc IterativeNewtonStep}}

\def\Or{\mathcal{O}_{\mathcal{K}}}

\newtheorem{lemma}{Lemma}

\newtheorem{corollary}{Corollary}
\newtheorem{theorem}{Theorem}
\newtheorem{proposition}{Proposition}

\usepackage[numbers]{natbib}

\title{Faster Convex Optimization: \\ Simulated Annealing with an Efficient Universal Barrier }

\author{
Jacob Abernethy  \\
Department of Computer Science\\
University of Michigan \\
\texttt{jabernet@umich.edu} \\
\and
Elad Hazan  \\
Department of Computer Science\\
Princeton University \\
\texttt{ehazan@princeton.edu} \\
}

\begin{document}

\maketitle

\begin{abstract}
This paper explores a surprising equivalence between two seemingly-distinct convex optimization methods. We show that simulated annealing, a well-studied random walk algorithms, is {\it directly equivalent}, in a certain sense, to the central path interior point algorithm for the the entropic universal barrier function. This connection exhibits several benefits. First, we are able improve the state of the art time complexity for convex optimization under the membership oracle model. We improve the analysis of the randomized algorithm of \citet{kalai2006simulated} by utilizing tools developed by \citet{nesterov1994interior} that underly the central path following interior point algorithm. We are able to tighten the temperature schedule for simulated annealing which gives an improved running time, reducing by square root of the dimension in certain instances. Second, we get an efficient randomized interior point method with an efficiently computable universal barrier for any convex set described by a membership oracle. Previously, efficiently computable barriers were known only for particular convex sets.
\end{abstract}

\newpage

\section{Introduction}

Convex optimization is by now a well established field and a cornerstone of the fields of algorithms and machine learning. Polynomial time  methods for convex optimization belong to relatively few classes: the oldest and perhaps most general is the ellipsoid method with roots back to Kachiyan in the 50s \citep{grotschel1993geometric}. Despite its generality and simplicity, the ellipsoid method is known to perform poorly in practice.

A more recent family of algorithms are the celebrated interior point methods, initially developed by Karmarkar in the context of linear programming, and generalized in the seminal work of \citet{nesterov1994interior}. These methods are known to perform well in practice and come with rigorous theoretical guarantees of polynomial running time, but with a significant catch: the underlying constraints must admit an efficiently-computable \emph{self-concordant barrier function}. Barrier functions are defined as satisfying certain differential inequality conditions that facilitate the path-following scheme developed by \citet{nesterov1994interior}, in particular it guarantees that the Newton step procedure maintains feasibility of the iterates. Indeed the iterative path following scheme essentially reduces the optimization problem to the construction of a barrier function, and in many nice scenarios a self-concordant barrier is easy to obtain; e.g., for polytopes the simple logarithmic barrier suffices. Yet up to the present there is no known universal {\it efficient} construction of a barrier that applies to any convex set. The problem is seemingly even more difficult in the \emph{membership oracle model} where our access to $\K$ is given only via queries of the form ``is $x \in \K$?''.

The most recent polynomial time algorithms are random-walk based methods, originally pioneered in the work of \citet{Dyer91} and greatly advanced by \citet{lovasz2006fast}. These algorithms apply in full generality of convex optimization and require only a membership oracle. The state of the art in polynomial time convex optimization is the random-walk based algorithm of simulated annealing and the specific temperature schedule analyzed in the breakthrough of \citet{kalai2006simulated}. Improvements have been given in certain cases, most notably in the work of \citet{narayanan2010random} where barrier functions were utilized.

In this paper we tie together two of the three known methodologies for convex optimization, give an efficiently computable universal barrier for interior point methods, and derive a faster algorithm for convex optimization in the membership oracle model. 
Specifically, 
\begin{enumerate}

\item

We define the {\bf heat path} of a simulated annealing method as the (determinisitc) curve formed by the mean of  the annealing distribution as the temperature parameter is continuously decreased.   We show that the heat path coincides with the {\bf central path} of an interior point algorithm with the entropic universal barrier function. This intimately ties the two major convex optimization methods together and shows they are approximately equivalent over {\it any} convex set.  

We further enhance this connection by showing that the central path following interior point method applied with the universal entropic barrier is a first-order approximation of simulated annealing. 

\item
Using the connection above, we give an efficient randomized interior point method with an efficiently computable universal barrier for any convex set described by a membership oracle. Previously, efficiently computable barriers were known only for particular convex sets.

\item
We give a new temperature schedule for simulated annealing inspired by interior point methods. This gives rise to an algorithm for general convex optimization with running time of $\tilde{O}(\sqrt{\nu} n^4 )$, where $\nu$ is the self-concordance parameter of the entropic barrier. The previous state of the art is $\tilde{O}(n^{4.5})$ by \cite{kalai2006simulated}. We note that our algorithm does not need explicit access to the entropic barrier, it only appears in the analysis of the temperature schedule.

It was recently shown by \citet{bubeck2014entropic} that the entropic barrier satisfies all require self-concordance properties and that the associated barrier parameter satisfies $\nu \leq n(1 + o(1))$, although this is generally not tight. Our algorithm improves the previous annealing run time by a factor of  $\tilde{O}(\sqrt{\frac{n}{\nu}})$ which in many cases is $o(1)$. For example, in the case of semi-definite programming over matrices in $\reals^{n \times n}$, the entropic barrier is identically the standard log-determinant barrier \cite{guler1996barrier}, exhibiting a parameter $\nu = n$, rather than $n^2$, which an improvement of $n$ compared to the state-of-the-art. More details are given in section \ref{sec:entbarrierhist}.

\end{enumerate}

\paragraph*{The Problem of Convex Optimization}
For the remainder of the paper, we will be considering the following algorithmic optimization problem. Assume we are given access to an arbitrary bounded convex set $\K \subset \reals^n$, and we shall assume without loss of generality that $\K$ lies in a 2-norm ball of radius 1. Assume we are also given as input a vector $\tht \in \reals^n$. Our goal is to solve the following:  
\begin{equation} \label{eq:obj}
\min_{x \in \K} \tht^\top x.
\end{equation}
We emphasize that this is, in a certain sense, the most general convex optimization problem one can pose. While the objective is linear in $x$, we can always reduce non-linear convex objectives to the problem \eqref{eq:obj}. If we want to solve $\min_{x \in \K} f(x)$ for some convex $f : \K \to \reals$, we can instead define a new problem as follows. Letting $\K' := \{ (x,c) \in \K \times \reals : f(x) - c \leq 0 \}$, this non-linear problem is equivalent to solving the linear problem $\min_{\{ (x,c) \in \K'\}} c$. This equivalence is true in in the membership oracle model, since a membership oracle for $\K$ immediately implies an efficient membership oracle for $\K'$.

\subsection{Preliminaries} 
This paper ties together notions from probability theory and convex analysis, most definitions are deferred to where they are first used. We try to follow the conventions of interior point literature as in the excellent text of \citet{nemirovski1996interior}, and the simulated annealing and random-walk notation of \cite{kalai2006simulated}.
 
For some constant $C$, we say a distribution $P$ is $C$-\emph{isotropic} if for any vector $v \in \reals^d$ we have 
$$\frac 1 C \|v \|^2 \leq \mathop{\E}_{X \sim P} [(v^\top X)^2] \leq C \|v\|^2 .$$
Let $P, P'$ be two distributions on $\reals^n$ with means $\mu, \mu'$, respectively. We say $P$ is $C$-isotropic with respect to $P'$ if 
$$\frac 1 C \E_{X \sim P'} [(v^\top X)^2] \leq \E_{X \sim P} [(v^\top X)^2] \leq C \E_{X \sim P'} [(v^\top X)^2] . $$ 
One measure of the distance between two distributions, often referred to as the $\ell_2$ norm, is given by  
$$ \left \| \frac{\mu}{\pi} \right \| \equiv \E_{ x \sim \mu} \left  ( \frac{\mu(x) }{\pi(x) } \right )  = \int_{ x \sim \mu} \left  ( \frac{\mu(x) }{\pi(x) } \right ) d \mu(x).  $$
We note that this distance is not symmetric in general.

For a differentiable convex function $f : \reals^n \to \reals$, the \emph{Bregman divergence} $D_f(x,y)$ between points $x,y \in \text{dom}(f)$ is the quantity 
$$  D_f(x,y) \equiv f(x) - f(y) - \nabla f(y)^\top(x - y) . $$
Further, we can always define the \emph{Fenchel conjugate} (or \emph{Fenchel dual}) \citep{rockafellar1970convex} which is the function $f^*$ defined as 
\begin{equation}
	\textstyle f^*(\theta) := \sup_{x \in \text{dom}(f)} \theta^\top x - f(x).
\end{equation}
It is easy to see that $f^*(\cdot)$ is also convex, and under weak conditions one has $f^{**} = f$. A classic duality result (see e.g. \cite{rockafellar1970convex}) states that when $f^*$ is smooth and strictly convex on its domain and tends to infinity at the boundary, we have a characterization of the gradients of $f$ and $f^*$ in terms of maximizers:
\begin{equation}\label{eq:fenchelgradient}
	\nabla f^*(\theta) = \mathop{\arg\max}_{x \in \text{dom}(f)} \theta^\top x - f(x) \quad \text{ and } \quad 
		\nabla f(x) = \mathop{\arg\max}_{\theta \in \text{dom}(f^*)} \theta^\top x - f^*(\theta).
\end{equation}

\subsection{Structure of this paper}

We start by an overview of random-walk methods for optimization in the next section, and introduce the notion of the {\it heat path} for simulated annealing. 
The following section surveys the important notions from interior point methods for optimization and the entropic barrier function. In section \ref{sec:equivalence} we tie the two approaches together formally by proving that the heat path and central path are the same for the entropic barrier. We proceed to  give a new temperature schedule for simulated annealing as well as prove its convergence properties. In the appendix we describe the Kalai-Vempala methodology for analyzing simulated annealing and its main components for completeness.

\section{An Overview of Simulated Annealing}

Consider the following distribution over the set $\K$ for an arbitrary input vector $\theta \in \reals^n$. 
\begin{equation}
\textstyle
	P_\theta(x) := \frac{\exp(-\theta^\top x)}
	{ \int_K \exp(-\theta^\top x' ) \, dx' }.
\end{equation}
This is often referred to as the \emph{Boltzmann distribution} and is a natural exponential family parameterized by $\theta$. It was observed by \cite{kalai2006simulated} that the optimization objective \eqref{eq:obj} can be reduced to sampling from these distributions. That is, if we choose some scaling quantity $t > 0$, usually referred to as the {\it temperature}, then any sample $X$ from the distribution $P_{\tht/t}$ must be $ n t$-optimal in expectation. More precisely, \cite{kalai2006simulated} show that
\begin{equation} \label{eq:kvapproxerror}
	 \mathop{\mathbb{E}}_{X \sim P_{\tht/t  }}[\tht^\top X] - \min_{x \in \K} \tht^\top x \leq  n t.
\end{equation}

As we show later, our connection implies an even stronger statement, replacing $n$ above by the self-concordant parameter of the entropic barrier, as we will define in the next section equation \eqref{eq:xberror}.

It is quite natural that for small temperature parameter $t \in \reals$, samples from the $P_{\frac{ \theta}{t}}$ are essentially optimal -- the exponential nature of the distribution will eventually concentrate all probability mass on a small neightborhood around the maximizing point $x^* \in \K$. The problem, of course, is that sampling from a point mass around $x^*$ is precisely as hard as finding $x^*$.

This brings us to the technique of so-called \emph{simulated annealing}, originally introduced by \citet{kirkpatrick1983optimization} for solving generic problems of the form $\min_{x \in \K} f(x)$, for arbitrary (potentially non-convex) functions $f$. At a very high level, simulated annealing would begin by sampling from a ``high-entropy'' distribution ($t$ very close to $0$), and then continue by slowly ``turning down the temperature'' on the distribution, i.e. decreasing $t$, which involves sampling according to the pdf $Q_{f,t}(x) \propto \exp(- \frac{1}{t} f(x))$. The intuition behind annealing is that, as long as $t'/t$ is a small constant, then the distributions $Q_{f,t'}$ and $Q_{f,t}$ will be ``close'' in the sense that a random walk starting from the initial distribution $Q_{f,t'}$ will mix quickly towards its stationary distribution $Q_{f,t}$.

Since its inception, simulated annealing is generally referred to as a heuristic for optimization, as polynomial-time guarantees have been difficult to establish. However, the seminal work of \citet{kalai2006simulated} exhibited a poly-time annealing method with formal guarantees for solving linear optimization problems in the form of \eqref{eq:obj}. Their technique possessed a particularly nice feature: the sampling algorithm utilizes a well-studied random walk (Markov chain) known as \har\ \citep{smith1984efficient,lovasz1999hit,lovasz2006fast}, and the execution of this Markov chain requires only access to a \textbf{membership oracle} on the set $\K$. That is, \har\ does not rely on a formal description of $\K$ but only the ability to (quickly) answer queries ``$x \in \K$?'' for arbitrary $x \in \reals^d$.

Let us now describe the \har\ algorithm in detail. We note that this Markov chain was first introduced by \citet{smith1984efficient}, a poly-time guarantee was given by \citet{lovasz1999hit} for uniform sampling, and this was generalized to arbitrary log-concave distributions by \citet{lovasz2003geometry}. \har\ requires several inputs, including: (a) the distribution parameter $\theta$, (b) an estimate of the covariance matrix $\Sigma$ of $P_\theta$, (c) the membership oracle $\Or$, for $\K$, (d) a starting point $X_0$, and (e) the number of iterations $N$ of the random walk.

\begin{algorithm}
	\textbf{Inputs:} parameter vector $\theta$, oracle $\Or$ for $\K$, covariance matrix $\Sigma$, \#iterations $N$, initial $X_0 \in \K$. \\
	\For{$i = 1, 2, \ldots, N$}{
		Sample a random direction $u \sim N(0, \Sigma)$ \label{line:direction}\\
		Querying $\Or$, determine the line segment $R = \{X_{i-1} + \rho u \; : \; \rho \in \reals \} \cap \K$ \label{line:makeray}\\
		Sample a point $X_i$ from $R$ according to the distribution $P_\theta$ restricted to $R$ \label{line:sampleray}
	}
	Return $X_N$\\
\caption{$\har(\theta, \Or, N, \Sigma, X_0)$ for approximately sampling $P_\theta$}
\end{algorithm}

\iffullversion

\begin{figure}[t]
\centering
\includegraphics[width=0.6\textwidth ]{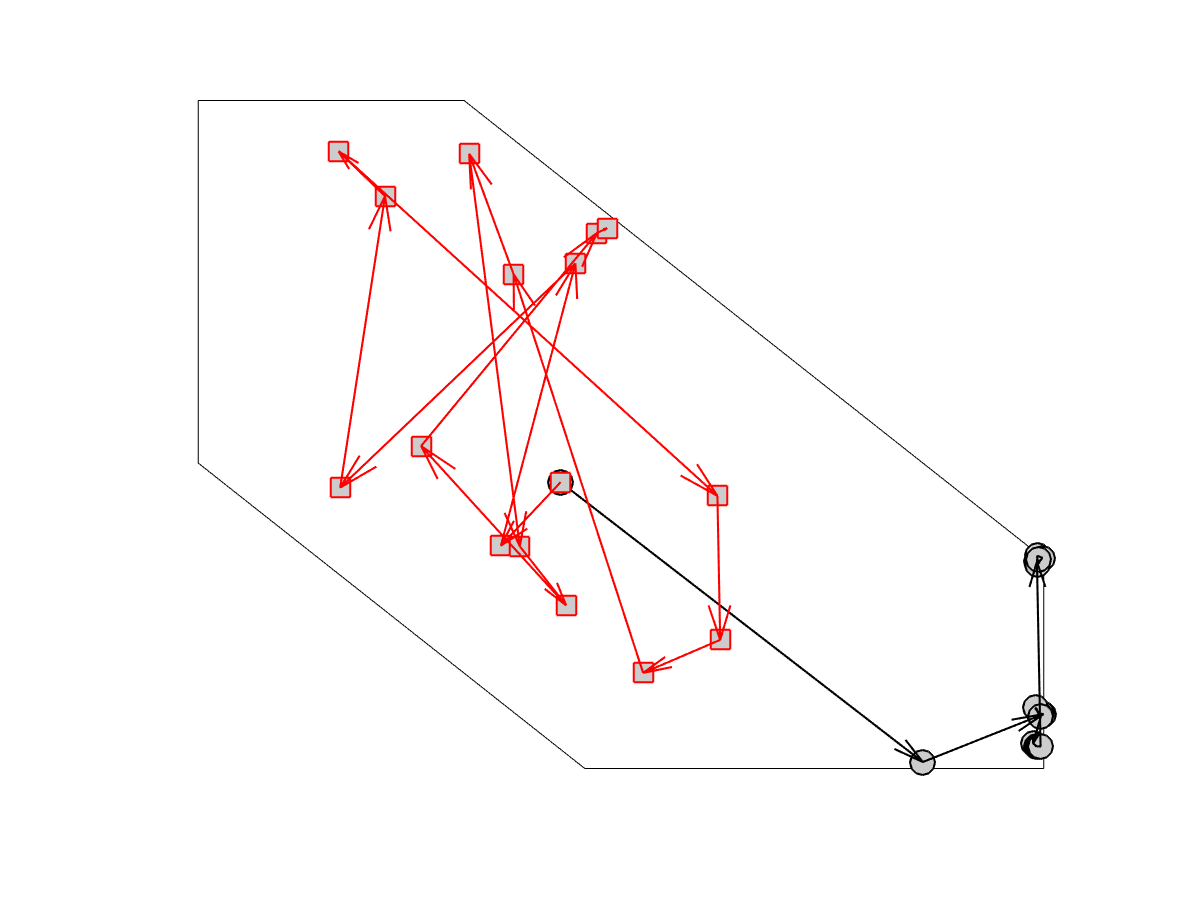}
\caption{The progression of two Hit-and-Run random walks for a high temperature (red squares) and a low temperature (black circles). Notice that at low temperature the walk coverges very quickly to a corner of $\K$.}\label{hitNrun}
\end{figure}

\fi

The first key fact of $\har(\theta)$ is that the stationary distribution of this Markov chain is indeed the desired $P_\theta$; a proof can be found in \cite{vempala2005geometric}. The difficulty for this and many other random walk techniques is to show that the Markov chain ``mixes quickly'', in that the number of steps $N$ needn't be too large as a function of $n$. This issue has been the subject of much research will be discussed below. Before proceeding, we note that a single step of \har\ can be executed quite efficiently. Sampling a random gaussian vector with covariance $\Sigma$ (line \ref{line:direction}) can be achieved by simply sampling a standard gaussian vector $z$ and returning $\Sigma^{1/2}z$. Computing the line segment $R$ (line \ref{line:makeray}) requires simply finding the two locations where the line $\{X_{i-1} + \rho u \; : \; \rho \in \reals \}$ intersects with the boundary of $\K$, but an $\epsilon$-approximation of these points can be found via binary search using $O\left(\log \frac 1 \epsilon\right)$ queries to $\Or$. Sampling from $P_\theta$ restricted to the line segment $R$ can also be achieved efficiently, and we refer the reader to \citet{vempala2005geometric}.

The analysis for simulated annealing in \cite{kalai2006simulated} proceeds by imagining a sequence of distributions $P_{\theta_k}  = P_{{\tht}/{t_k}}$ where $t_1=R$ is the diameter of the set $\K$ and $t_k := \left(1 - \frac 1 {\sqrt{n}}\right)^{k} $. Let $k = O(\sqrt{n}\log \frac n \epsilon )$, then sampling from $P_{\theta_k}$ is enough to achieve the desired optimization guarantee. That is, via Equation~\ref{eq:kvapproxerror}, we see a sample from $P_{\theta_k}$ is $\epsilon$-optimal in expectation.

To sample from $P_{\theta_k}$, \cite{kalai2006simulated} construct a recursive sampling oracle using \har. The idea is that samples from $P_{\theta_{k + 1}}$ can be obtained from a warm start by sampling from $P_{\theta_{k}}$ according to a carefully chosen temperature schedule. The details are given in Algorithm~\ref{alg:sa}.
\begin{algorithm}
Input: temperature schedule $\{t_k, k \in [T]\}$, objective $\tht$.  \\
Set $X_0 = 0$, $\Sigma_1 = I$, $t_1 = R$ \\
\For{$k = 1,...,T$}{
 $\theta_{k} \leftarrow \frac{\tht}{t_k} $   \\
 $X_{k} \leftarrow \har(\theta_k, \Or, N, \Sigma_k, X_{k-1}) $\\
	\For  {$j= 1,..,n$}{
	 $Y^j_{k} = \har(\theta_k, \Or, N, \Sigma_k, Y^j_{k-1}) $
	}
Estimate covariance: $\Sigma_{k+1} := \textnormal{CovMtx}(Y_1^k, \ldots, Y_n^k)$
}
Return $X_T$
\caption{\textsc{SimulatedAnnealing with HitAndRun} -- \citet{kalai2006simulated} \label{alg:sa}}
\end{algorithm}

The \citet{kalai2006simulated} analysis leans on a number of technical but crucial facts which they prove. The temperature update schedule that they devise, namely $t_k = (1-\frac{1}{\sqrt{n}}) t_{k-1}$, is shown to satisfy these iterative rules and thus return an approximate solution.

\begin{theorem}[Key result of \citet{kalai2006simulated} and \citet{lovasz2003geometry}]  \label{thm:KVmain}
Fix $k$ and consider the $\har$ walk used in Algorithm \ref{alg:sa} to compute $X_k$ and $Y^j_k$ for each $j$. Assume we choose the temperature schedule in order that
successive distributions $P_{\theta_k}, P_{\theta_{k-1}}$ are close in $\ell_2$:
	\begin{equation} \label{eqn:kv-cond}
\textstyle  \max\left\{  \left\| \frac{P_{\theta_k}}{P_{\theta_{k-1}}} \right\|_2 ,  \left\|  \frac{P_{\theta_{k-1}}}{P_{\theta_{k}}}   \right\|_2 \right\}  \leq 10.
	\end{equation}
Then, as long as the warm start samples $X_{k-1}$ and $Y^j_{k-1}$ are (approximately) distributed according to $P_{\theta_{k-1}}$, the random walk \har\ mixes to $P_{\theta_k}$ with $N = \tilde O( n^3 )$ steps. That is, the output samples $X_k$ and $Y^j_k$ are distributed according to $P_{\theta_k}$ up to error $\leq \epsilon$.
 \end{theorem}
In the appendix we sketch the proof of this theorem for completeness. 
\begin{corollary}
The temperature schedule  $t_k  := \left(1 - 1/\sqrt{n}\right)^{k} t_1 $ satisfies condition \eqref{eqn:kv-cond}, and thus Algorithm \ref{alg:sa} with this schedule returns an $\epsilon$-approximate solution in time $\tilde{O}(n^{4.5})$. 
\end{corollary}
\begin{proof}
By equation \eqref{eq:kvapproxerror}, to achieve $\epsilon$ error it suffice that $\frac{1}{t} \geq \frac{n}{\epsilon}$, or in other words $k$ needs to be  large enough such that $(1 - \frac{1}{\sqrt{n}})^{k} \leq \frac{\epsilon}{n}$ for  which $k = 8 \sqrt{n} \log \frac{n}{\epsilon}$ suffices: $(1 - \frac{1}{\sqrt{n}})^{k}  \leq e^{- \frac{k}{2 \sqrt{n}}} = e^{ - 4 \log \frac{n}{\epsilon} } \leq \frac{\epsilon}{n}$.  
Hence the temperature schedule need be applied with $T = \tilde{O}(\sqrt{n})$ iterations.   Each iteration requires $O(n)$ applications of \har\ that cost $O(n^3)$, for the total running time  of $\tilde{O}(n^{4.5})$.
\end{proof}
In later sections we proceed to give a more refined temperature schedule that satisfies the Kalai-Vempala conditions, and thus results in a faster algorithm. Our temperature schedule is based on new observations in interior point methods, which we describe next.  

\subsection{The heat path for simulated annealing}

Our main result follow from the observation that the path-following interior point method has an analogue in the random walk world. Simulated annealing incorporates a carefully chosen temperature schedule  to reach its objective from a near-uniform distribution. We can think of all temperature schedules as performing a random process whose changing mean is a single well-defined curve. For a given convex set $\K \subseteq \reals^d$ and objective $\tht$, define the heat path as the following set of points, parametrized by the temperature $t \in [0,\infty]$ as follows:
$$  \hp(t)  =   \E_{ x \sim P_{\tht/t}} [x]  $$
We can now define the heat path as $\hp  = \mathop{\cup}_{t \geq 0}  \{ \hp(t) \}$. 
At this point it is not yet clear why this set of points is even a continuous curve in space, let alone equivalent to an analogous notion in the interior point world. We will return to this equivalence in section \ref{sec:equivalence}.

\section{An Overview of Interior Point Methods for Optimization}

Let us now review the vast literature on Interior Point Methods (IPMs) for optimization, and in particular the use of the Iterative Newton Step technique. The first instance of polynomial time algorithms for convex optimization using interior point machinery was the linear programming algorithm of \citet{karmarkar}. The pioneering book of \citet{nesterov1994interior} brought up techniques in convex analysis that allowed for polynomial time algorithms for much more general convex optimization, ideas that are reviewed in great detail and clarity in  \cite{nemirovski1996interior} .

The goal remains the same, to solve the linear optimization problem posed in Equation~\eqref{eq:obj}. The intuition behind IPMs is that iterative update schemes such as gradient descent for solving \eqref{eq:obj} can fail because the boundary of $\K$ can be difficult to manage, and ``moving in the direction of descent'' will fail to achieve a fast rate of convergence. Thus one needs to ``smooth out'' the objective with the help of an additional function. In order to produce an efficient algorithm, a well-suited type of function is known as a \emph{self-concordant barrier} which have received a great deal of attention in the optimization literature.

A self-concordant barrier function $\varphi : \text{int}(\K) \to \reals$, with \emph{barrier parameter}, $\nu$ is a convex function satisfying two differential conditions: for any $h \in \reals^n$ and any $x\in \K$,
\begin{eqnarray}
	\nonumber \nabla^3\varphi[h,h,h] & \leq & 2 (\nabla^2\varphi[h,h])^{3/2}, \text{ and }\\
	\label{eq:barrierdef} \nabla \varphi[h] & \leq & \sqrt{\nu \nabla^2\varphi[h,h]},
\end{eqnarray}
in addition to the property that the barrier should approach infinity when approaching the boundary of $\K$. 
Such function possess very desirable properties from the perspective of optimization, several of which we discuss in the present section. We note an important fact: while the existence of such a function $\varphi$ for general sets $\K$ has been given by \citet{nesterov1994interior}---the \emph{universal barrier} with parameter parameter $\nu = O(n)$, to be discussed further in Section~\ref{sec:entbarrierhist}---an efficient construction of such a function has remained elusive and was considered an important question in convex optimization. This indeed suggests that the annealing results we previously outlined are highly desirable, as \har\ requires only a membership oracle on $\K$. However, one of the central results of the present work is the equivalence between annealing and IPMs, where we show that sampling gives one implicit access to a particular barrier function. This will be discussed at length in Section~\ref{sec:equivalence}.

Let us now assume we are given such a function $\varphi$ with barrier parameter $\nu$. A standard approach to solving \eqref{eq:obj} is to add the function $\varphi(x)$ to the primary objective, scaling the linear term by a ``temperature'' parameter $t > 0$:
\begin{equation} \label{eq:scobj}
	\min_{x \in \K}\;  \{ t  \tht^\top x + \varphi(x) \}.
\end{equation}
As the the temperature $t$ tends to $\infty$ the solution of \eqref{eq:scobj} will tend towards the optimal solution to \ref{eq:obj}. This result is proved for completeness in Theorem~\ref{thm:main1}.

Towards developing in detail the iterative Newton algorithm, let us define the following for every positive integer $k$:
\begin{eqnarray}
	t_k & := & \textstyle \left( 1 + \frac{c}{\sqrt{\nu}} \right)^k \quad \quad \text{for some $c>0$,} \label{eq:alphadef}  \\
	\Phi_k(x) & := & t_k \tht^\top x + \varphi(x) \nonumber \\
	\bx_k & := & \arg\min_x \Phi_k(x)  \nonumber
\end{eqnarray}
As $\varphi$ is a barrier function, it is clear that $\bx_k$ is in the interior of $K$ and, in particular, $\nabla \Phi_k(\bx_k) = 0 \implies \nabla \varphi(\bx_k) = t_k \tht$. It is shown in \cite{nemirovski1996interior} (Equation 3.6) that any $\nu$-SCB (Self-Concordant Barrier) $\varphi$ satisfies $\nabla \varphi(x)^\top(y - x) \leq \nu$, whence we can bound the difference in objective value between $\bx_k$ and the optimal point $x^*$:
\begin{equation}\label{eq:xberror}
	\tht^\top (x^* - \bx_k) = \frac{\nabla \varphi(\bx_k)^\top(x^* - \bx_k)}{t_k}  \leq \frac{\nu}{t_k}.
\end{equation}
We see that the approximation point $\bx_k$ becomes exponentially better as $k$ increases. Indeed, setting $k = \lceil \frac{\sqrt{\nu}}{c}\log(\nu/\epsilon)\rceil$ guarantees that the error is bounded by $\epsilon$.

The iterative Newton's method technique actually involves approximating $\bx_k$ with $\hx_k$, a near-optimal maximizer of $\Phi_k$, at each iteration $k$. For an arbitrary $v \in \reals^n$, $x \in \text{int}(\K)$, and any $k \geq 1$, following \cite{nemirovski1996interior} we define:
\begin{align}
	 \quad \| v \|_x \; := \; & \sqrt{v^\top \nabla^2 \varphi(x) v},
	 	& \quad \quad & \text{the ``local norm'' of $v$ w.r.t $x$;}\\
	\| v \|^*_x \; := \; & \sqrt{v^\top \nabla^{-2} \varphi(x) v},
		& \quad \quad & \text{the corresponding dual norm of $v$,}\\
	\lambda(x,t_k) \; := \; & \| \nabla \Phi_k(x) \|^*_x,
		& \quad \quad & \text{the \emph{Newton decrement} of $x$ for temperature $t_k$.}
 \end{align}
Note that, for a fixed point $x \in K$, the norms $\| \cdot \|_x$ and $\| \cdot \|_x^*$ are dual to each other\footnote{Technically, for $\| \cdot \|_x$ and its dual to be a norm, we need $\nabla^2 \varphi$ to be positive definite and $\varphi$ to be strictly convex. One can verify this is the case for bounded sets, which is the focus of this paper.}. It will turn out that $\lambda(x,t_k)$ will be used both as a quantity in the algorithm, and as a kind of potential that we need to keep small.

In Algorithm~\ref{alg:newton} we describe the damped newton update algorithm, henceforth called \newton. We note that the 
\begin{algorithm}
	\textbf{Input:} $\tht \in \reals^d$, $\K$ and barrier function $\varphi$\\
	\textbf{Solve:} $\hx_0 = \arg\max_{x \in \K} \tht^\top x + \varphi(x)$\\
	\For{$k = 1, 2, \ldots$}{
		$\hx_k \leftarrow \hx_{k-1} - \frac{1}{1 + \lambda(\hx_{k-1},t_k)} \nabla^{-2} \varphi(\hx_{k-1}) \nabla \Phi_{k}(\hx_{k-1})$
	}
\caption{\newton} \label{alg:newton}
\end{algorithm}

The most difficult part of the analysis is in the following two lemmas, which are crucial elements of the \newton\ analysis. A full exposition of these results is found in the excellent survey \cite{nemirovski1996interior}. The first lemma tells us that when we update the temperature, we don't perturb the Newton decrement too much. The second lemma establishes the \emph{quadratic convergence} of the Newton Update for a fixed temperature.
\begin{lemma} \label{lem:tempupdate}
	Let $c$ be the constant chosen in the definition \eqref{eq:alphadef}. Let $t > 0$ be arbitrary and let $t' = t\left(1 + \frac c {\sqrt{\nu}}\right)$. Then for any $x \in \text{int}(\K)$, we have $\lambda(x, t') \leq (1 + c)\lambda(x, t) + c$.
\end{lemma}
\begin{lemma}\label{lem:quadraticconvergence}
	Let $k$ be arbitrary and assume we have some $\hx_{k-1}$ such that $\lambda(\hx_{k-1},t_{k})$ is finite. The Newton update $\hx_{k}$ satisfies $\lambda(\hx_k,t_k) \leq 2 \lambda^2(\hx_{k-1},t_{k})$.
\end{lemma}
The previous two lemmas can be combined to show that the following invariant is maintained. Neither the constant bound of $1/3$ on the Newton decrement nor the choice of $c = 1/20$ are particularly fundamental; they are convenient for the analysis but alternative choices are possible.
\begin{lemma}\label{lem:lambdabound}
	Assume we choose $c = 1/20$ for the parameter in \eqref{eq:alphadef}. Then for all $k$ we have $\lambda(\hx_k,t_k) < \frac 1 3$.
\end{lemma}
\begin{proof}We give a simple proof by induction. The base case is satisfied since we assume that $\lambda(\hx_{0},t_{0}) = 0$, as $t_0 = 1$.\footnote{As stated, Algorithm~\ref{alg:newton} requires finding the minimizer of $\varphi(\cdot)$ on $\K$, but this is not strictly necessary. The convergence rate can be established as long as a ``reasonable'' initial point $\hx_{0}$ can be computed---e.g. it suffices that $\lambda(\hx_0,1) < 1/2$.} For the inductive step, assume $\lambda(\hx_{k-1},t_{k-1}) < 1/3$. Then
\begin{eqnarray*}
	\lambda(\hx_{k},t_k) & \leq & 2\lambda^2(\hx_{k-1},t_k)\\
	& \leq & 2((1 + c)\lambda(\hx_{k-1},t_{k-1}) + c)^2  < 2(0.4)^2 < 1/3.
\end{eqnarray*}
The first inequality follows by Lemma~\ref{lem:quadraticconvergence} and the second by Lemma~\ref{lem:tempupdate}, hence we are done.
\end{proof}
\begin{theorem} \label{thm:main1}
	Let $x^*$ be a solution to the objective \eqref{eq:obj}. For every $k$, $\hx_k$ is an $\epsilon_k$-approximate solution to \eqref{eq:obj}, where $\epsilon_k = \frac{\nu + \sqrt{\nu}/4}{t_k}$.  In particular, for any $\epsilon > 0$, as long as $k > \frac{\sqrt{\nu}}{c}\log(2\nu/\epsilon)$ then  $\hx_k$ is an $\epsilon$-approximation solution.
\end{theorem}
\begin{proof} Let $k$ be arbitrary. Then,
	\begin{eqnarray*} 
		\tht^\top( \hx_k - x^*) & = & \tht^\top(\bx_k - x^*) + \tht^\top(\hx_k - \bx_k) \\
		\text{\tiny (By \eqref{eq:xberror})} \hspace{1in} 
			& \leq & \textstyle  \frac{\nu}{t_k} + \tht^\top(\hx_k - \bx_k) \\
		\text{\tiny (H\"older's Inequality)} \hspace{1in} 
			& \leq & \textstyle  \frac{\nu}{t_k} + \|\tht\|_{\bx_k}^* \|\bx_k - \hx_k\|_{\bx_k} \\
		\text{\tiny (\cite{nemirovski1996interior} Eqn. 2.20)} \hspace{1in} 
			& \leq & \textstyle  \frac{\nu}{t_k} + \|\tht\|_{\bx_k}^* \frac{\lambda(\hx_k, t_k)}{1 - \lambda(\hx_k, t_k)} \\
		\text{\tiny (Applying Lemma~\ref{lem:lambdabound} and \eqref{eqn:grad-barrier})} \hspace{1in}
			& \leq & \textstyle  \frac{\nu}{t_k} + \left\|\frac {\nabla \varphi(\bx_k)}{t_k}\right\|_{\bx_k}^* \frac 1 4  \quad \leq \quad \frac{\nu + \sqrt{\nu}/4}{t_k}
	\end{eqnarray*}
The last equation utilizes a fact that derives immediately from the definition \eqref{eq:barrierdef}, namely
\begin{equation} \label{eqn:grad-barrier}
	\| \nabla \varphi^*(x) \|_x^* = \| \nabla \varphi^*(x) \|_{\theta(x)} \leq \sqrt{\nu}
\end{equation}
holds for any SCBF $\varphi$ with parameter $\nu$, and for any $x \in \K$.
\end{proof}
We proceed to give a specific barrier function that applies to any convex set and gives rise to an efficient algorithm.

\section{The Equivalence of IterativeNewton and SimulatedAnnealing} \label{sec:equivalence}

We now show that the previous two techniques, Iterative Newton's Method and Simulated Annealing, are in a certain sense two sides of the same coin. In particular, with the appropriate choice of barrier function $\varphi$ the task of computing the sequence of Newton iterates $\hx_1, \hx_2, \ldots$ may be viewed precisely as estimating the means for each of the distributions $P_{\theta_1}, P_{\theta_2}, \ldots$.  This correspondence helps to unify two very popular optimization methods, but also provides three additional novel results:
\begin{enumerate}
	\item We show that the heat path for simulated annealing is equivalent to the central path with the entropic barrier. 

	\item We show that the running time of Simulated Annealing can be improved to $O(n^4\sqrt{\nu})$ from the previous best of $O(n^{4.5})$. In the most general case we know that $\nu = O(n)$, but there are many convex sets in which the parameter $\nu$ is significantly smaller. Notably such is the case for the positive-semi-definite cone, which is the basis of semi-definite programming and a cornerstone of many approximation algorithms. Further examples are surveyed in section \ref{sec:entbarrierhist}.
	
	\item We show that Iterative Newton's Method, which previously required a provided barrier function on the set $\K$, can now be executed efficiently where the only access to $\K$ is through a membership oracle. This method relies heavily on previously-developed sampling techniques \citep{kalai2006simulated}. Discussion is deferred to Appendix \ref{sec:ipm-sampling}.
\end{enumerate}

\subsection{The Duality of Optimization and Sampling} \label{sec:duality}

\iffullversion

We begin by rewriting our Boltzmann distribution for $\theta$ in exponential family form,
\begin{equation} \label{eq:redefptheta}
	\textstyle P_\theta(x) := \exp(- \theta^\top x - A(\theta)) \quad \text{ where } \quad A(\theta) := \log \int_K \exp(-\theta^\top x')dx'.
\end{equation}
The function $A(\cdot)$ is known as the \emph{log partition function} of the exponential family, and it has several very natural properties in terms of the mean and variance of $P_\theta$:
\begin{eqnarray}
	\nabla A(\theta) & = & - \E_{X \sim P_\theta}[X] \label{eq:deriv_expectation} \\
	\nabla^2 A(\theta) & = & \E_{X \sim P_\theta}[(X - \E X)(X - \E X)^\top].
\end{eqnarray}
We can also appeal to Convex (Fenchel) Duality \citep{rockafellar1970convex} to obtain the conjugate
\begin{equation} \label{eq:fenchelA}
	\textstyle A^*(x) := \sup_{\theta} \theta^\top x - A(\theta).
\end{equation}
It is easy to establish that $A^*$ is smooth and strictly convex. The domain of $A^*(\cdot)$ is precisely the space of gradients of $A(\cdot)$, and it is straightforward to show that this is the set $\text{int}(-\K)$, the interior of the reflection of $\K$ about the origin. However we want a function defined on (the interior of) $\K$, not its reflection, so let us define a new function $A^*_-(x) := A^*(-x)$ whose domain is $\text{int}(\K)$. We now present a recent result of \citet{bubeck2014entropic}.
\begin{theorem}[\cite{bubeck2014entropic}]
	The function $A^*_-$ is a $\nu$-self-concordant barrier function on $\K$ with $\nu \leq n(1 + o(1))$.
\end{theorem}
One of the significant drawbacks of barrier/Newton techniques is the need for a readily-available self-concordant barrier function. In their early work on interior point methods, \citet{nesterov1994interior} provided such a function, often referred to as the ``universal barrier'', yet the actual construction was given implicitly without oracle access to function values or derivatives. \citet{bubeck2014entropic} refer to function $A^*_-(\cdot)$ as the \emph{entropic barrier}, a term we will also use, as it relates to a notion of differential entropy of the exponential family of distributions.

It is important to note that, when our set of interest is a cone $K$, the entropic barrier on $K$ corresponds exactly to the Fenchel dual of the universal barrier on the dual cone $K^*$ \cite{guler1997self}, which immediately establishes the self-concordance. Indeed, many beautiful properties of the entropic barrier on cones have been developed, and for a number of special cases $A^*_-(\cdot)$ corresponds exactly to the canonical barrier used in practice; e.g. $A^*_-(\cdot)$ on the PSD cone corresponds to the log-determinant barrier. In many such cases one obtains a much smaller barrier parameter $\nu$ than the $n(1 + o(1))$ bound. We defer a complete discussion to Section~\ref{sec:entbarrierhist}.

In order to utilize $A^*_-(\cdot)$ as a barrier function as in \eqref{eq:scobj} we must be able to approximately solve objectives of the form $\min_{x \in \K} \{ \theta^\top x + A^*_-(x) \}$. One of the key observations of the paper, given in the following Proposition, is that solving this objective correponds to computing the mean of the distribution $P_\theta$.
\begin{proposition} \label{prop:connection}
	Let $\theta \in \reals^n$ be arbitrary, and let $P_\theta$ be defined as in \eqref{eq:redefptheta}. Then we have
	\begin{equation}
		\E_{X \sim P_\theta}[X] = \mathop{\arg\min}_{x \in \textnormal{int}(\K)} \left\{ \theta^\top x + A^*_-(x) \right\}.
	\end{equation}
\end{proposition}
\begin{proof}
	Let $\theta$ be an arbitrary input vector. As we mentioned in \eqref{eq:fenchelgradient}, standard Fechel duality theory gives us
	\[
		\nabla A(\theta) = \mathop{\arg\max}_{x \in \text{dom}(A^*)} \left\{  \theta^\top x - A^*(x) \right\} = \mathop{\arg\min}_{x \in \text{dom}(A^*)}  \left\{ -\theta^\top x + A^*(x) \right\} .
	\]
	It can be verified that the domain of $A^*$ is precisely the interior of $-\K$, the reflection of $\K$. Thus we have
	\[
		\nabla A(\theta) = \mathop{\arg\min}_{x \in \text{int}(-\K)} \left\{ -\theta^\top x + A^*(x) \right\} = - \left(\mathop{\arg\min}_{x \in \text{int}(\K)} \left\{ \theta^\top x + A^*_-(x) \right\} \right).
	\]
	In addition, we noted in \eqref{eq:redefptheta} that $\nabla A(\theta) = - \E_{X \sim P_\theta}[X]$. Combining the latter two facts gives the result.
\end{proof}
We now have a direct connection between sampling methods and barrier optimization. For the remainder of this section, we shall assume that our chosen $\varphi(\cdot)$ is the entropic barrier $A^*(\cdot)$, and the quantities $\Phi_k(\cdot), \| \cdot \|_x, \lambda(\cdot,\cdot)$ are defined accordingly. We shall also use the notation $x(\theta) := \E_{X \sim P_\theta}[X] = -\nabla A(\theta)$. 
\begin{lemma} \label{lem:divdec}
	Let $\theta,\theta'$ be such that $\|x(\theta') - x(\theta)\|_{x(\theta)} \leq \frac 1 2$. Then we have 
	\begin{equation}
		\textstyle
		D_{A^*_-}(x(\theta'), x(\theta))
		= \text{KL}(P_\theta',P_\theta)
		= D_A(\theta,\theta')
		\leq 2 \|x(\theta') - x(\theta)\|^2_{x(\theta)}
	\end{equation}
\end{lemma}
\begin{proof}
	The duality relationship of the Bregman divergence, and its equivalence to Kullback-Leibler divergence, is classical and can be found in, e.g., \cite{wainwright2008graphical} equation (5.10)  The final inequality follow as a direct consequence of \cite{nemirovski1996interior}, Equation 2.4.
\end{proof}

\fi

\subsection{Equivalence of the heat path and central path}

\iffullversion

The most appealing observation on the equivalence between random walk optimization and interior point methods is the following geometric equivalence of curves. For a given convex set $\K \subseteq \reals^d$ and objective $\tht$, define the heat path as the following set of points:
$$ \hp  = \mathop{\cup}_{t \geq 0}  \{ \hp(t) \} = \mathop{\cup}_{t \ge 0}   \{  \E_{ x \sim P_{ \frac{ \tht } {t} }} [x] \} $$

To see that this set of points is a continuous curve in space, consider the central path w.r.t. barrier function $\varphi(x)$: 
$$ \cp(\varphi)  = \mathop{\cup}_{t \geq 0}  \{ \cp(t,\varphi) \} = \mathop{\cup}_{t \ge 0}   \{ \arg 	\min_{x \in \K}\;  \{ t  \tht^\top x + \varphi(x) \} $$

It is well known that the central path is a continuous curve in space for any self-concordant barrier function $\phi$. We now have the following immediate corollary of Proposition \ref{prop:connection}:
\begin{corollary}\label{cor:hpcp}
The central path corresponding to the self-concordant barrier $A^*$ over any set $\K$ is equivalent to the heat path over the same set, i.e. 
$$ \hp \equiv \cp(A^*)$$ 
\end{corollary} 

This mathematical equivalence is demonstrated in figure \ref{heatpath} generated by simulation over a polytope. 

\begin{figure}[t]
\centering
\includegraphics[width=0.6\textwidth ]{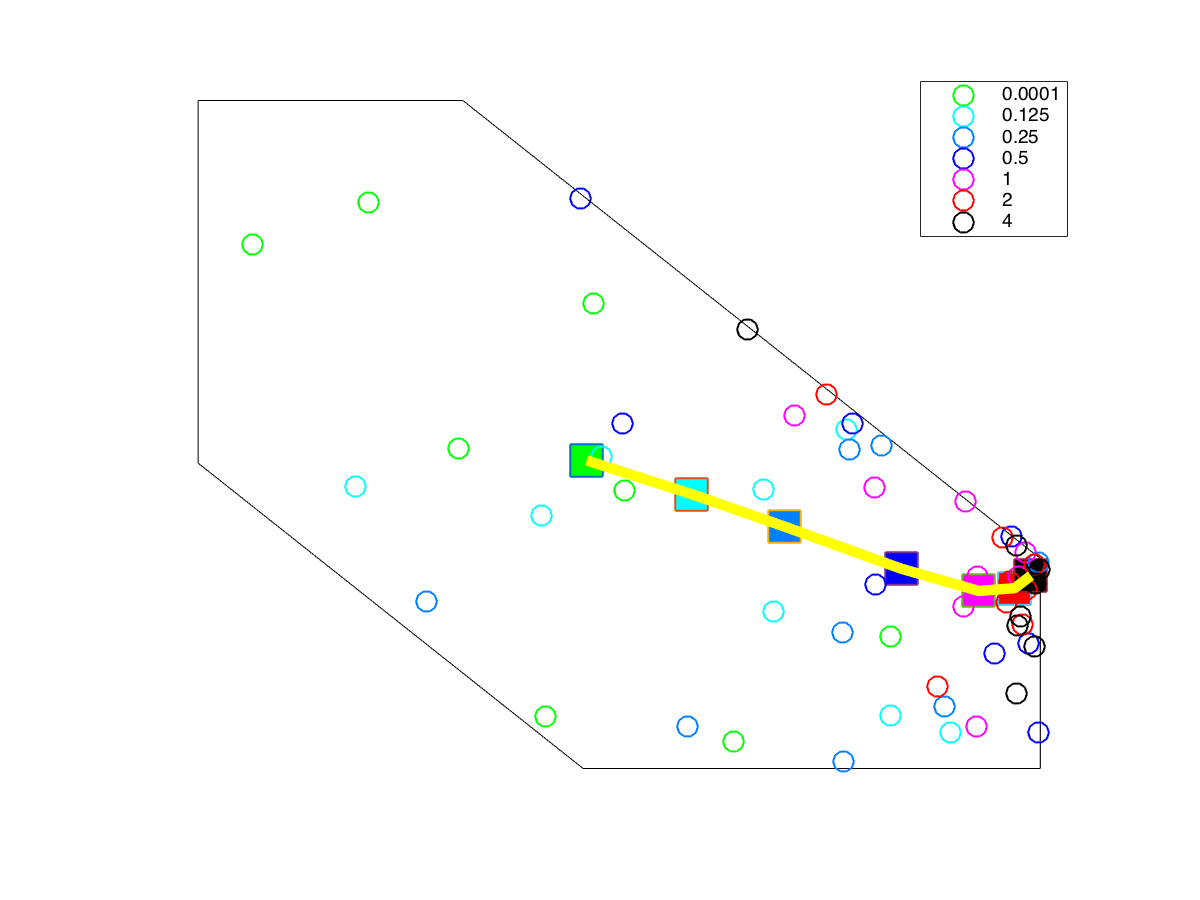}
\caption{For a set of seven different temperatures $t$, we used Hit-and-Run to generate and scatter plot several samples from $P_{\theta/t}$ using colored circles. We also computed the true means for each distribution, plotted with squares, giving a curve representing the $\hp$ across the seven temperatures. Of course via  Corollary~\ref{cor:hpcp} this corresponds exactly to the $\cp$ for the entropic barrier.}\label{heatpath}
\end{figure}

\fi

\subsection{IPM techniques for sampling and the new schedule}

\iffullversion

We now prove our main theorem, formally stated as:
\begin{theorem} \label{thm:ourmain}
	The temperature schedule of $\theta_1 = R$ where $R = \diam(\K)$ and $\theta_k := \left(1 - \frac 1 {4 \sqrt{\nu}}\right) \theta_{k-1}$, for $\nu$ being the self-concordance parameter of the entropic barrier for the set $\K$, satisfies condition \eqref{eqn:kv-cond} of theorem \ref{thm:KVmain}. Therefore  algorithm \ref{alg:sa} with this schedule returns an $\epsilon$-approximate solution in time $\tilde{O}(\sqrt{\nu} n^{4})$. 
\end{theorem}

Condition \eqref{eqn:kv-cond} is formally proved in the following lemma, which crucially uses the interior point methodology, namely Lemma \ref{lem:lambdabound}. 
\begin{lemma}
	Consider distributions $P_\theta$ and $P_{\theta'}$ where $\theta' = (1 + \gamma )\theta$ for $\gamma < \frac{1}{6 \sqrt{\nu}}$. Then we have the following bound on the $\ell_2$ distance between distributions:
		\begin{equation*} 
			\max\left\{  \left\| \frac{P_{\theta}}{P_{(1 + \gamma) \theta}} \right\|_2 ,  \left\|  \frac{P_{(1+\gamma) \theta}}{P_{\theta}}   \right\|_2 \right\} \quad \leq 10
		\end{equation*}
		
\end{lemma}
\begin{proof}
	We first show by elementary linear algebra that  
	\begin{equation*}
			\| P_\theta/P_{(1 - \gamma)\theta} \|  = \| P_\theta/P_{(1 + \gamma)\theta} \| = \exp(D_A((1+\gamma)\theta, \theta) + D_A((1-\gamma)\theta, \theta)).
		\end{equation*}
		Let us consider the $\log$ of the 2-norm:  \eh{isn't there a minus sign here everywhere??} 
	\begin{eqnarray*}
		\log \| P_\theta/P_{(1 + \gamma)\theta} \| 
		& = & \log \int_{\K} \frac{\exp(-\theta^\top x - A(\theta))}{\exp(-(1+\gamma)\theta^\top x - A((1+\gamma)\theta))} dP_\theta \\
		& = & \log \int_{\K} \exp(\gamma \theta^\top x - A(\theta) + A((1+\gamma)\theta)) dP_\theta \\
		& = & \log \int_{\K} \exp(\gamma \theta^\top x - A(\theta) + A((1+\gamma)\theta)) \exp(-\theta^\top x - A(\theta)) dx \\
		& = & A((1+\gamma)\theta) - 2A(\theta) + \log \int_{\K} \exp(-(1-\gamma)\theta^\top x )dx  \\
		& = &A((1+\gamma)\theta) - 2A(\theta) + A((1-\gamma)\theta)  \\
		& = & D_A((1+\gamma)\theta, \theta) + D_A((1-\gamma)\theta, \theta).
	\end{eqnarray*}
Replacing $\gamma$ by $-\gamma$, we get a completely symmetrical expression.
Next, we observe that 
$$ 	\| P_{\theta(1+\gamma)}  /P_{\theta} \| =  \| P_{\tilde{\theta}}  /P_{\tilde{\theta}( 1 - \tilde{\gamma}) } \|  $$
where $\tilde{\theta}  = \theta(1 + \gamma)$ and $1 - \tilde{\gamma} = \frac{1}{1 + \gamma}  = 1 - \frac{\gamma}{1 + \gamma} $, thus $\tilde{\gamma} \in \gamma \times [1 , 2] $.  By this observation, both sides of the lemma follow if we prove an upper bound 
$$\left\| \frac{P_{\theta}}{P_{(1 + \gamma) \theta}} \right\|_{2} \leq 10 \quad \text{  for    } \quad \gamma <  \frac{1}{6 \sqrt{\nu} } \times 2 = \frac{1}{3 \sqrt{\nu} } $$

	Lemma~\ref{lem:tempupdate} implies $\lambda(x(\theta),1 + \gamma) \leq (1 + c) \lambda(x(\theta),1) + c =  c \le \frac{1}{4}$. Applying Lemma~\ref{lem:divdec}, 
	\begin{eqnarray*}
		D_A((1+\gamma)\theta, \theta) & \leq  2 \|x(\theta) - x((1+\gamma)\theta)\|^2_{x((1+\gamma)\theta)} & \mbox{ Lemma \ref{lem:divdec} } \\
		& \leq  2 \left(\frac{ 	\lambda(x(\theta),1 + \gamma) 	}{ 	1-\lambda(x(\theta),1 + \gamma) 	} \right)^2  & \mbox { (2.28) in \citep{nemirovski1996interior}}  \\
		&  \leq 2 \left(\frac{c}{1-c}\right)^2 < 2 \frac{(1/3)^2}{(2/3)^2}  = \frac 1 2  & \mbox{ Lemma \ref{lem:tempupdate}} 
	\end{eqnarray*}
Notice that to apply Lemma \ref{lem:divdec}, we need the point $x((1 + \gamma)\theta) $ to lie in the Dikin ellipsoid of $x(\theta)$, which is exactly whats proved in the last two lines of the above proof.  	
	
	The bound on $D_A((1-\gamma)\theta, \theta)$ follows in precisely the same fashion, by similar change of variables as before (again, the condition for applying Lemma \ref{lem:divdec} is proven in the last few lines of the equations below):
	\begin{eqnarray*}
		D_A((1-\gamma)\theta, \theta) & =  D_A(  \tilde{\theta}, \tilde{\theta} ( 1 + \tilde{\gamma} )) \\
		& \leq  2 \|x( \tilde{\theta} ) - x((1+\tilde{\gamma}) \tilde{\theta})\|^2_{x( \tilde{\theta} )}  & \mbox { Lemma \ref{lem:divdec} } \\
		& \leq  2 \left(\frac{ 	\lambda(x(\tilde{\theta}),1 + \tilde{\gamma}) 	}{ 1-\lambda(x(\tilde{\theta}),1 + \hat{\gamma}) 	} \right)^2 & \mbox{ (2.27) in \citep{nemirovski1996interior}} \\
		&  \leq 2 \left(\frac{c}{1-c}\right)^2 < 2 \frac{(1/3)^2}{(2/3)^2}  = \frac 1 2 & \mbox{ Lemma \ref{lem:tempupdate}} 
	\end{eqnarray*}
	
	It follows that 
	$\| P_\theta/P_{(1 + \gamma)\theta} \|  \leq  e^{ D_A((1+\gamma)\theta, \theta) + D_A((1-\gamma)\theta, \theta)} \leq e^{\frac{1}{2} + \frac{1}{2}}  \leq 10$.
\end{proof}

\fi

\subsection{Some history on the entropic barrier and the universal barrier for cones} \label{sec:entbarrierhist}

\iffullversion

Let $K$ be a cone in $\reals^n$ and let $K^* = \{ \theta : \theta^\top x \geq 0\; \forall\, x \in K \}$ be its dual cone. We note that a cone $K$ is \emph{homogeneous} if its automorphism group is transitive; that is, for every $x,y \in K$ there is an automorphism $B : K \to K$ such that $Bx = y$. Homogeneous cones are very rare, but one notable example is the PD cone (matrices with all positive eigenvalues). Given any point $x \in K$, we can define a truncated region of $K^*$ as the set $K^*(x) := \{ y \in K^*: x^\top y \leq 1 \}$. \citet{nesterov1994interior} defined the first generic self-concordant barrier function, known as the \emph{universal barrier} in terms of these regions. Namely, they show that the function 
\[
	u_K(x) := \log(\vol(K^*(x)))
\]
is a self concordant barrier function with an $O(n)$ parameter.

There is an alternative characterization of the universal barrier in terms of the larg partition function. Let $F_{K}(x) := \log \int_{K^*} \exp(\theta^\top x)d\theta$ and equivalently let $F_{K^*}(\theta) := \log \int_{K} \exp(\theta^\top x)dx$. It was shown by \citet{guler1996barrier} that
\[
	F_{K}(x) = u_K(x) + n!,
\]
that is, the universal barrier corresponds exactly to a log-partition function but defined on \emph{the dual cone} $K^*$, modulo a simple additive constant. We note that this is not the entropic barrier construction we have here, as our function of interest is $A^*_-(\cdot) \equiv F^*_{K^*}(\cdot)$ (the Fenchel conjugate of $F_{K^*}(x)$), and not $F_K(x)$. However, it was also shown by \citet{guler1996barrier} that, when $K$ is a homogeneous cone, we have the identity $F_K(\cdot) \equiv F^*_{K^*}(\cdot)$; in other words, the universal barrier and the entropic barrier are equivalent for homogeneous cones.

It is worth noting that, following the connection of \citet{guler1996barrier}, $A^*_-(\cdot)$ is (up to additive constant) the Fenchel conjugate of the universal barrier $u_{K^*}$ for $K^*$. It was shown by \citet{nesterov1994interior} (Theorem 2.4.1) that Fenchel conjugation preserves all required self concordance properties and in particular if $g$ is a $\nu$-self-concordant barrier for a cone $K$, then $g^*$ will be a self-concordant barrier for $K^*$ with the same parameter $\nu$. With this observation it follows immediately that the entropic barrier $F^*_{K^*}(\cdot)$ on $K$ is an $O(n)$-self-concordant barrier. \citet{bubeck2014entropic} took this statement further, proving that $F^*_{K^*}(\cdot)$ enjoys an essentially optimal self-concordance parameter $\nu = n(1 + o(1))$.

It is important to note that the assumption that the set of interest is a cone is, roughly speaking, without loss of generality. Given any convex set $U \subset \reals^n$ we have the \emph{fitted cone} $K(U) := \{ \alpha(x,1) : x \in U, \alpha \geq 0 \} \subseteq \reals^{n+1}$. Hence once can always work with the barrier function $F_{K(U)^*}^*(\cdot)$ on $K(U)$, and take its restriction to the set $U\times\{1\} \subset K(U)$ to obtain a barrier on $U$ (affine restriction preserves the barrier properties).

We conclude by summarizing several results in \citet{guler1996barrier} regarding the entropic barrier for various cones, as well as the associated barrier parameter of each. In these canonical cases the entropic barrier corresponds exactly to the ``typical'' barrier, up to additive and multiplicative constants. We use the notation $f(\cdot) \cong g(\cdot)$ to denote that $f$ and $g$ are identical up to additive constants.
\begin{enumerate}
	\item Assume $K := \reals^n_+$ the nonnegative orthant. This is a homogeneous cone and we have $F_K(x) \cong F_{K^*}^*(x) \cong -\sum_{i=1}^n \log x_i$. This is the optimal barrier for $K$ and the barrier parameter is $\nu = n$.
	\item Assume $K := \{ x \in \reals^n: x_1^2 + \ldots + x_{n-1}^2 \leq x_n^2 \}$ be the \emph{Lorentz cone}. $K$ is a homogeneous self-dual cone.  $K$ can also be described by the fitted cone of the radius-1 $L2$ ball, so we may parameterize elements of $K$ as $\alpha(x,1)$ where $\alpha \geq 0$ and $x$ is vector in $\reals^{n-1}$ with $L2$ norm bounded by 1. Then $F_K(\alpha(x,1)) \cong F_{K^*}^*(\alpha(x,1)) \cong - \frac{n+1}{2} \log(\alpha^2 - x^2)$. This barrier has parameter $\nu = n+1$ which is indeed not optimal, one has the optimal barrer $-\log(\alpha^2 - x^2)$ which has parameter $\nu = 2$, but this is simply a scaled version of the entropic barrier.
	\item The PSD cone $K$ of positive semi-definite matrices, i.e. symmetric matrices with non-negative eigenvalues, is a homogeneous self-dual cone. The entropic barrier is $F_K(x) \cong F_{K^*}^*(x) \cong - \frac{n+1}{2} \log \text{det}(x)$ and exhibits a parameter of $\nu = \frac{n(n+1)}{2}$ which is multiplicatively $\frac{n+1}{2}$ worse than the optimal barrier, the log-determined $-\log \text{det}(x)$. However, as can be seen this barrier is quite simply a scaled version of the entropic barrier.
\end{enumerate}

\fi


\bibliographystyle{abbrvnat}
\bibliography{sampling}



\appendix

\section{An Explanation of Newton's Method via Reweighting} \label{sec:reweighting}

\iffullversion

Proposition~\ref{prop:connection} brings out a strong connection between interior point techniques and the ability to sample from Boltzmann distributions. But with this stochastic viewpoint, it is not immediately clear why Newton's method is an appropriate iterative update scheme for optimization. We now provide some evidence along these lines.

Assuming we have already computed (an approximation of) $x(\theta)$, and our distribution parameter is updated to a ``nearby'' $\theta'$, our goal is now to compute the new mean $x(\theta')$.
\begin{eqnarray*}
-x(\theta') = \int_\K x \, dP_{\theta'} = \int_\K x \, \frac{dP_{\theta'}(x)}{dP_{\theta}(x)}dP_{\theta} & = & \E_{X \sim P_\theta}\left[ X \exp(X^\top (\theta - \theta^{\prime}) + A(\theta') - A(\theta)) \right]
\end{eqnarray*}
Think of the last term as the reweighting factor. Now we are going to rewrite $A(\theta) - A(\theta') = - \nabla A(\theta) (\theta' - \theta) - D_A(\theta',\theta) =  x(\theta)^\top(\theta' - \theta) - \text{KL}(P_\theta,P_{\theta'})$. We shall use the following approximation of the exponential: $\exp(z) \approx 1 + z$ for small values of $z$. Proceeding, 
\begin{eqnarray*}
	- x(\theta') & = & \E_{X \sim P_\theta}\left[ X \exp(X^\top (\theta - \theta^{\prime}) - x(\theta)^\top(\theta' - \theta) + \text{KL}(P_\theta,P_{\theta'})) \right] \\
	& = & e^{\text{KL}(P_\theta,P_{\theta'})} \E_{X \sim P_\theta}\left[ X \exp((X - x(\theta))^\top (\theta^{\prime} - \theta)) \right] \\
	& \approx & e^{\text{KL}(P_\theta,P_{\theta'})} \E_{X \sim P_\theta}\left[ X (1 + (X - x(\theta))^\top (\theta^{\prime} - \theta)) \right] \\
	& = & e^{\text{KL}(P_\theta,P_{\theta'})} (- x(\theta) + \Sigma_\theta (\theta' - \theta)). 
\end{eqnarray*}
Duality theory tells us that $\Sigma_\theta = \nabla^2A(\theta) = \nabla^{-2} A^*(x(\theta))$ and $\theta' - \theta$ is precisely the gradient of the objective $\theta^{\prime \top} x - A^*(x)$ at the point $x(\theta)$. The $e^{\text{KL}(P_\theta,P_{\theta'})}$ term is somewhat mysterious, but it can be interpreted as something of a ``damping'' factor akin to the Newton decrement damping of the the Newton update.

\section{Proof structure of the Kalai-Vempala theorem}

We hereby sketch the structure of the proof of theorem \ref{thm:KVmain} for completeness. 
Recall the statement of the theorem: 

Algorithm \ref{alg:sa} with a temperature schedule that satisfies the following condition:\\

 The successive distributions are not ``too far'' in total variational distance. That is, for every $j$, \\
$$ \max\left\{  \left\| \frac{P_{\theta_j}}{P_{\theta_{j-1}}} \right\|_2 ,  \left\|  \frac{P_{\theta_{j-1}}}{P_{\theta_{j}}}   \right\|_2 \right\}   \leq  10 $$

Guarantees that \har\ requires $N = \tilde O( n^3 )$ steps in order to ensure mixing to the stationary distribution $P_{\theta_{j}}$.

\begin{proof}[Proof sketch]
The proof is based on iteratively applying the following Theorem from \cite{LovaszV06}:

\begin{theorem} \label{thm:LV-old}
Let f be a density proportional to $e^{-a^T x} $  over a convex set K such that 
\renewcommand{\theenumi}{[\alph{enumi}]}
\begin{enumerate}
\item
the level set of
probability 1/64 contains a ball of radius s
\item
$ \E_f (\| x - \mu_f\|^2 ) \leq S$, where $\mu_f = \E_f[x]$ is the mean of $f$
\item
the $\ell_2$ norm of the starting
distribution $\sigma$ w.r.t. the stationary distribution of \har\ denoted $\pi_f$,  is at most M. 
\end{enumerate}
Let $\sigma^m$ be the distribution of the current
point after m steps of \har\ applied to f. Then, for any $\tau > 0$, after $ m = O(\frac{n^2 S^2}{s^2} \log^5 \frac{nM}{\tau} )$
steps, the total variation distance of $\sigma^m$ and $\pi_f$ is less than $\tau$. 
\end{theorem}

The proof proceeds to show that conditions [a]-[c] of the theorem above are all satisfied if indeed condition \eqref{eqn:kv-cond} is satisfied, along the  steps below. Once it is established that the conditions of the theorem hold, then the next \har\ walk mixes and computes warm start and variance estimates for the next epoch. Then again, the conditions of the theorem hold, and this whole process is repeated for the entire temperature schedule. 

To show conditions [a]-[c], first notice that condition \eqref{eqn:kv-cond} is essentially equivalent to condition [c] above. Thus we only need to worry about conditions [a],[b].

\renewcommand{\theenumi}{[\Roman{enumi}]}
\begin{enumerate}
\item
For simplicity, we  assumed that at the current iteration, $\Sigma_j = I $ is the identity. This can be assumed by a transformation of the space, and allows for simpler discussion of isotropy of densities (otherwise, the isotropy condition would be replaced by relative-isotropy w.r.t the current variance). 
	
\item
A density $f$ is C-isotropic if for any unit vector $\|v\| = 1$, 
$$ \frac{1}{C} \leq \int_{\reals^n} (v^\top (x - \mu_f))^2 f(x) dx \leq C$$
It is shown (Lemma 4.2) that if the density given by $f$  is $O(1)$-isotropic, then conditions [a],[b] are satisfied with $\frac{S}{s} = \tilde{O}(\sqrt{n})$. 

\item
It is shown (Lemma 4.3) that if $f$ is C-isotropic, and $  \max\left\{  \left\|\frac{f}{g}\right\|_2 ,  \left\|\frac{g}{f}\right\|_2 \right\} \leq D$, then $g$ is $CD$-isotropic.  

\item
Since condition \eqref{eqn:kv-cond} holds, together with the previous points [II,III] this implies that $f_{\theta_{j+1}}$ is isotropic for some constant. Thus,  conditions [a]-[c] of Theorem \ref{thm:LV-old} hold. Therefore we can sample sufficiently many samples to estimate the covariance matrix $\Sigma_{j+1}$ and proceed to the next epoch.   
\end{enumerate}

Throughout the proof special care needs to be taken to ensure that repeated samples are nearly-independent for various concentration lemmas to apply, we omit discussion of these and the reader is referred to the original paper of \cite{kalai2006simulated}. 
	
\end{proof}

\fi

\section{Interior point methods with a membership oracle} \label{sec:ipm-sampling}

\iffullversion

Below we sketch a universal IPM algorithm - one that applies to any convex set described by a membership oracle - that can be implemented to run in polynomial time. This algorithm is an instantiation of Algorithm \ref{alg:newton} with the particular barrier function $A^*(x)$ as defined in section \ref{sec:duality}.

Without loss of generality, we can assume our goal is to (approximately) compute the update direction
\[
	\nabla^{-2}A^*(x)(\theta - \nabla A^*(x))
\]
for some $x$ which is already within the Dikin ellipsoid of radius 1/2 around $x(\theta)$. First, we note that the IPM analysis of \cite{nemirovski1996interior} allows one to replace the inverse hessian $\nabla^{-2}A^*(x)$ with the nearby $\nabla^{-2}A^*(x(\theta)) = \text{CovMtx}(P_\theta)$. Of course the latter can be estimated via sampling, in the sense that the estimate $\hat \Sigma$ will be ``$\epsilon$-isotropically close'':
$$(1-\epsilon)v^\top \nabla^2 \Psi(\theta') v \leq v^\top  \hat \Sigma v \leq (1+\epsilon)v^\top \nabla^2 \Psi(\theta') v $$
for any unit vector $v$. See, for example, \cite{adamczak2010quantitative} on the concentration of empirical covariance matrices.


It remains to compute $\nabla A^*(x)$.
Define $\theta(x)$ to be 
\begin{equation}
	\theta(x) = \arg\max_{\theta} \theta\cdot x - \log \int_K \exp(-\theta \cdot y) dy = \nabla A^*(x) 
\end{equation}
Then $\theta(x)$ can be computed in polynomial time by another interior point algorithm -- this problem, however, is much simpler to work with. Define $\Psi(\theta') := \theta\cdot x - \log \int_K \exp(-\theta \cdot y) dy$ to be the objective we want to optimize. Notice that $\nabla \Psi(\theta') = x - \E_{X' \sim P_{\theta'}}[X']$ and the latter can be estimated to within $\epsilon$ via \textsc{SimulatedAnnealing} with $\tilde{O}(n/\epsilon^2)$ samples. The hessian $\nabla^2 \Psi(\theta') = - \text{CovMtx}(P_{\theta'})$ can similarly be estimated with an $\epsilon$-isotropically close empirical covariance. Because the error gap is multiplicatively close to 1, the inverse operation on $\nabla^2 \Psi(\theta')$ maintains the approximation.

\fi

\end{document}